\documentclass{gtpart}

\usepackage{pinlabel}
\title[THE INTERSECTION GRAPH OF AN ORIENTABLE GENERIC SURFACE]{THE INTERSECTION GRAPH OF AN ORIENTABLE GENERIC SURFACE}

\author[D\,Ben Hadar]{Doron Ben Hadar}
\givenname{Doron}
\surname{Ben Hadar}
\address{Department of Mathematics\\
Bar-Ilan university\\\newline
Ramat Gan\\IL}
\email{doron.ben@live.biu.ac.il}
\urladdr{math.biu.ac.il/node/641}


\arxivreference{math.GT/1412.4340}

\volumenumber{}
\issuenumber{}
\publicationyear{}
\papernumber{}
\startpage{}
\endpage{}
\doi{}
\MR{}
\Zbl{}
\received{}
\revised{}
\accepted{}
\published{}
\publishedonline{}
\proposed{}
\seconded{}
\corresponding{}
\editor{}
\version{}

\usepackage{pinlabel}
\usepackage{amssymb,amsfonts,amscd,eucal,eufrak,mathrsfs,amsmath}
\usepackage[all,arc]{xy}
\usepackage{enumerate}
\usepackage{mathrsfs}
\usepackage{hyperref} 
\usepackage{graphicx}

\theoremstyle{plain}
\newtheorem{thm}{Theorem}[section]

\newtheorem{lem}[thm]{Lemma}

\newtheorem{res}[thm]{Result}

\theoremstyle{definition}
\newtheorem{defn}[thm]{Definition}

\theoremstyle{remark}
\newtheorem{rem}[thm]{Remark}
\newtheorem{rems}[thm]{Remarks}

\def\R {{\mathbb R}}

\def\Z {{\mathbb Z}}

\makeatletter
\let\c@equation\c@thm

\numberwithin{equation}{section}

\begin{document} 

\begin{abstract}
The intersection graph $M(i)$ of a generic surface $i\co F \to S^3$ is the set of values which are either singularities or intersections. It is a multigraph whose edges are transverse intersections of two surfaces and whose vertices are triple intersections and branch values. $M(i)$ has an enhanced graph structure which Gui-Song Li referred to as a "daisy graph" (see \cite{Li1} p.3721.) If F is oriented then the orientation further refines $M(i)$'s structure into what Li called an "arrowed daisy graph".

Li left open the question "which arrowed daisy graphs can be realized as the intersection graph of an oriented generic surface?". The main theorem of this article will answer this. I will also provide some generalizations and extensions to this theorem in sections 4 and 5.
\end{abstract}

\begin{asciiabstract}
The intersection graph M(i) of a generic surface i\co F \to S^3 is the set of values which are either singularities or intersections. It is a multigraph whose edges are transverse intersections of two surfaces and whose vertices are triple intersections and branch values. M(i) has an enhanced graph structure which Gui-Song Li referred to as a "daisy graph" (see \cite{Li1} p.3721.) If F is oriented then the orientation further refines M(i)'s structure into what Li called an "arrowed daisy graph".

Li left open the question "which arrowed daisy graphs can be realized as the intersection graph of an oriented generic surface?". The main theorem of this article will answer this. I will also provide some generalizations and extensions to this theorem in sections 4 and 5.
\end{asciiabstract}

\maketitle

\section{Introduction - The Structure of the Intersection Graph}

A (proper) generic surface in a 3-manifold is a generalization of an immersed surface in a general position. Specifically:

\begin{defn} \label{generic}
A (proper) generic surface in a 3-manifold $M$ is a smooth mapping $i\co F \to M$, where $F$ is a compact surface (called the underlying surface) and each value of $i$ has a neighborhood $U$ in $M$ such that $U \cap i(F)$ looks like one of the pictures in figure \ref{fig:Figure 1} (The purple part is $\partial M$).

\begin{figure}
\begin{center}
\includegraphics{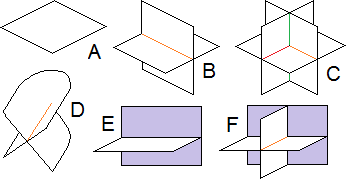}
\caption{The values of a generic surface}
\label{fig:Figure 1}
\end{center}
\end{figure}

[1A], [1B] and [1C] are (respectively) a regular value, a double value and a triple value. Locally they look like the intersection of 1, 2 or 3 of the coordinate planes in a neighborhood of the origin in $\R^3$. [1D] is a branch value, locally, it looks like the renowned "Whitney's Umbrella". [1E] is a regular boundary (RB) value and [1F] is a double boundary (DB) value, they resemble the intersection of 1 or 2 of the $[xz]$ and $[yz]$ planes, with the boundary (the $[xy]$ plane) in a neighborhood of the origin inside the upper half space.
\end{defn}

These surfaces are "proper" in the sense that $i(\partial F)=i(F) \cap \partial M$, and they are "generic" in the sense that every proper smooth function from a compact surface to $M$ can be turned into a proper generic surface via an arbitrarily small perturbation. They are also "stable" in the sense that a small enough perturbation can only change $i$ up to an isotopy of $F$ and $M$. If $F$ is closed, I call $i$ a "closed generic surface".

I am interested in the intersection graph of a generic surface. It is the set $M(i)=\overline{\{p \in i(F)| 1<|i^{-1}(p)|\}}$ of all values which are neither regular nor RB. I will regard the intersection graph in two ways, and each way has its own notations. I will also use specific conventions when I draw it.

\begin{defn} \label{consecutive}
1) First, I will regard the intersection graph as a multigraph whose vertices are the triple values (degree 6), branch values and DB values (degree 1) of $i$, and whose edges are the segments of a "double line" (a line consisting of double values, such as the orange line in figure \ref{fig:Figure 1}B) between two vertices. In addition to this "graph part", $M(i)$ may contain several "double circles" - double lines that close into circles instead of ending at vertices. Having no vertices or edges, double circles do not comply with the traditional definition of a graph, and need to be accounted for separately.

Note that $M(i)$ may have graph-theoretical "loops" - paths whose ends are both at the same vertex (which must be a triple value, as its degree is not $1$). Due to this, it will be important later on to distinguish between the two "ends of an edge". I use the term "half-edge" to describe such an end.

In figure \ref{fig:Figure 1}C I show that 3 "segments" of double line (marked in orange, red and green) intersect at each triple value. The intersection cuts each segment into a pair of half-edges that seem to be in continuation to one another. I say that such a pair of half edges is "consecutive". Each triple value has three disjoint pairs of consecutive half edges. Notice that these two half-edges may come originate in same edge, in which case said edge is a loop.

2) Secondly, I will regard a pair of consecutive half-edges to be parts of a long path that crosses the triple value. $M(i)$ is then the union of several of these long lines, which I call "double arcs". Three double arcs intersect at each triple value, but this number includes multiplicity - it may be that a double arc crosses itself at a triple value. Double arcs are thus immersed 1-manifolds in $M$, not embedded ones.

A double arc may, as in figures \ref{fig:Figure 1}D and F, end in a branch value or a DB value on each side of it. Otherwise, it may close up into a circle. I refer to arcs of the latter kind as "closed" and to the former type as "open".

According to the "double arc" notation, a double circle is just a closed double arc that does not pass any triple values. A closed double arc that passes a triple value only once is simply a loop whose two half-edges are consecutive.

3) When drawing a diagram of the intersection graph, make sure that each triple value looks like the intersection of three lines, as in figure \ref{fig:Figure 2}A and B. In this way one can see which half-edges are consecutive, and what the double arcs are. A "purple point" at an end of an edge symbolizes that this edge ends in a DB value, as opposed to a branch value. Double circles are, of course, drawn as circles disjoint from the graph part.
\end{defn}

\begin{figure}
\begin{center}
\includegraphics{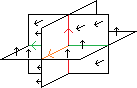}
\caption{Some examples of daisy graphs and arrowed daisy graphs}
\label{fig:Figure 2}
\end{center}
\end{figure}

Many authors have studied the intersection graph from various angels. For instance, in \cite{Izu&Mar1}, Izumiya and Marar showed a connection between the Euler char of the (image of the) closed generic surface, and the Euler char of its underlying surface. $\chi(i(F))=\chi(F)+T(i)+\frac{1}{2}B(i)$ where $T(i)$ is the number of triple values and $B(i)$ is the number of branch values of the surface (they called the latter "branch points", hence $B$). This generalized an earlier result by Carter and Ko (\cite{Car&Ko1}) which in turn generalized a result by Banchoff (\cite{Ban1}).

I am interested in the intersection graph because it encodes several important properties of the surface. For instance, Giller (\cite{Gil1}) showed that examining the intersection set can tell us if a generic surface in $\R^3$ can be lifted into an embedded surface in $\R^4$. Several other authors, including Carter and Saito (\cite{Car&Sai1}), and Satoh (\cite{Sat1}), have looked into the connection between liftings and the intersection graph.

One can, to some extent, classify generic surfaces according to their intersection graph. An early example of this can be found in \cite{Cro&Mar1}, where Cromwell and Marar classified the kind of surface in $\R^3$ that can have an intersection graph of a certain form (connected and with only one vertex which is a triple value).

In this article I address the case in which both the 3-manifold $M$ and the underlying surface $F$ are oriented. In this case, one can add more data to a diagram of the intersection graph. To see this, one must first recall the notion of a co-orientation:

\begin{defn} \label{CoOr}
A co-orientation on a generic surface $i\co F \to M$ is a continuous choice, for every non branch value $x$, of a normal vector in $T_{i(x)}M$ that is orthogonal to to $D_i(T_xF)$. If $M$ is oriented then there is a 1-1 correspondence between orientations on $F$ and co-orientations on $i$. It matches each orientation on $F$ with the normal $\overrightarrow{n}$ for which $(D_i(\overrightarrow{v_1}),D_i(\overrightarrow{v_2}),\overrightarrow{n})$ upholds the orientation of $M$ whenever $(\overrightarrow{v_1},\overrightarrow{v_2})$ upholds that of $F$. 
\end{defn}

I use co-orientations to indicate the orientation of a generic surface in illustrations. In particular, figure \ref{fig:Figure 3} depicts the neighborhood of a triple value. Notice that each of the three "double arc segments" that pass through a triple value a) consists of the intersection of two of the three planes that intersect in the value, and b) intersects the remaining plane transversally. The normal arrows on this last plane point toward one "preferred direction" on this arc or, equivalently, toward one of the two consecutive half-edges. I refer to the half-edge the arrow points toward as the "preferred" one of the two.

\begin{figure}
\begin{center}
\includegraphics{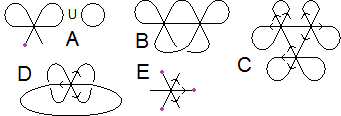}
\caption{The preferred directions at a triple value}
\label{fig:Figure 3}
\end{center}
\end{figure}

I can encode this information in a diagram of $M(i)$. I use a small arrow to mark the preferred direction on each intersecting arc segment at each triple value as in figures \ref{fig:Figure 2}C, D and E. I can formally define a type of combinatorial structure that encodes the relevant information about $M(i)$. This is a generalization of definitions made by Li in \cite{Li1} (p. 3721 and 3723).

\begin{defn} \label{Graph1}
1) A daisy graph (DG) $(V,E,n,B,C)$ is a 5-tuple where $(V,E)$ is multigraph whose vertices are all of degree 1 or 6. $n$ is a non-negative integer. $B$ is a subset of the set of degree-1 vertices, and, for each degree 6 vertices $v$, $C(v)$ is a division of the set of the half-edges of $v$ into three pairs.

For the DG of a generic surface, $n$ will indicate the number of double circles the surface has, $B$ is the set of DB values (the other degree 1 vertices are branch values), and for each triple value $v$, the three pairs in $C(v)$ are the three pairs of consecutive half-edges. In the interest of convenience, I call the vertices of any DG triple values, branch values, and DB values, in accordance with their degrees and belonging to $B$. I call a pair of half-edges in $C(v)$ consecutive. I draw a DG according to the conventions of definition \ref{consecutive}(3).

2) An arrowed daisy graph (ADG) $(V,E,n,B,C,A)$ is a 6-tuple where $(V,E,n,B,C)$ is a DG and for each triple value $v$ and each pair $p \in c(v)$ $A(v)$ is one of the half-edges in $p$, which I call the "preferred half-edge". In diagrams I mark each preferred half-edge with an arrow as in figures \ref{fig:Figure 2}C, D and E. In the ADG of an oriented generic surface I choose the preferred half-edges as per the co-orientation, as explained above.
\end{defn}

\begin{rem} \label{crossing}
1) Li assumed the surface is an immersion of a closed surface, so he did not have branch values or DB values.

2) Despite the similarity to graphs on surfaces, daisy graphs are not planar. One arc can go "above" another. I mark it as a crossing in a knot (See figures \ref{fig:Figure 2} B,D and E) to avoid confusion, but it is not a real crossing - it does not matter which arc is higher and which is lower.
\end{rem}

In \cite{Li1}, Li found out which DGs can be realized as the intersection graph of an orientable generic surface in $\R^3$. He defined the ADG in order to answer this, but this led to a new question: which ADGs can be realized as the intersection graph of an oriented generic surface in a given oriented 3-manifold $M$? (Li posted this as an open question, see \cite{Li1} p 3725). The main purpose of this article is to answer this question. It turns out that there are two inherently different cases. The case where $H(M;\Z)$ is periodic (all its elements have a finite order), and the cases where it is not. I solve the first case in sections 2 and 3 and the second case in section 4. In section 5, I will discuss a refinement of the notion of ADGs.

\begin{rem}
The solution I give in this paper is not fully constructive. Specifically, when I prove that an ADG is realizable, I only construct a part of the realizing surface. I then use arguments of homology and surgery to prove that this part can be extended to a whole surface. I have by now found several ways to construct an entire surface, but they are unneeded here, and will only lengthen the proof.

I will explain one of these constructions in a subsequent article, \cite{BH1}, where I will show that the problem of determining if a generic surface is liftable into an embedded surface in 4-space (a knotted surface) is NP complete. The purpose of the \cite{BH1} article only requires me to construct a realizing surface for a very specific type of realizable ADGs, but the construction given therein can be easily generalized to fit all constructible ADGs.
\end{rem}

\section{Gradings and Winding numbers}

\begin{defn} \label{GraphGrade}
1) Let $G$ be an ADG. I say that an edge $e$ is "preferred" (resp. "non-preferred") at a vertex $v$ if one of the ends of $e$ is a preferred (resp. non-preferred) half-edge at the vertex $v$.

2) A grading of an ADG $G$ is a choice of a number $g(e) \in \Z$ (called the grade of $e$) for every edge $e$ of $G$, such that, at each triple value $v$, all the non-preferred edges at $v$ have the same grade $a(v)$ and all preferred edges have the grade $a(v)+1$. An ADG that has a grading is called "gradable".

The grading concerns only the "graph part" of the ADG and ignores the double circles. Since double circles pose no obstruction to gradability, I consider an ADG that consists solely of double circles to be gradable.
\end{defn}

\begin{figure}
\begin{center}
\includegraphics{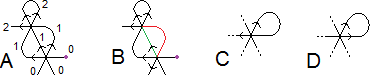}
\caption{The obstructions to gradability of an ADG}
\label{fig:Figure 4}
\end{center}
\end{figure}

Figure \ref{fig:Figure 4}A depicts a graded ADG. The ADG in \ref{fig:Figure 4}D is not gradable. The reason for this is that the red and green edges are both non-preferred at the upper triple value, implying that they ought to have the same grade, but at the bottom triple value one of them is preferred and the other is not, implying that they ought to have different grades. I will discuss the connection between the gradability and the realizability of an ADG shortly. However, I will being by explaining the factors that make an ADG gradable. 

\begin{defn} \label{GradeObs}
A "grade obstructing" loop of an ADG is a loop (a path whose ends both lay on the same vertex $v$) with one preferred end and one non-preferred end at $v$. For example, the loop in figure \ref{fig:Figure 4}C is grade obstructing, while the loop in figure \ref{fig:Figure 4}D is not.
\end{defn}

\begin{rem} \label{GradeObsRem}
1) A gradable ADG cannot have grade obstructing loops, since the grade of such a loop would have to be $a(v)=g(e)=a(v)+1$. 

2) If an ADG has no grade obstructing loops, then the sets of preferred edges at $v$ and non-preferred edges at $v$ are mutually exclusive. This simplifies the following definition.
\end{rem}

\begin{defn} \label{RelGrade}
1) Given a DADG $G$ with no grade obstructing loops, and two edges $e$ and $f$ that share a vertex $v$ (which must be a triple value since its degree cannot be 1), define the "grading difference" $\Delta g(e,v,f)$ to be $1$ if $f$ is preferred at $v$ and $e$ is not, $-1$ if it is the other way around, and $0$ if either both $f$ and $g$ are preferred or if they are both non-preferred.

2) The grading difference of a path $e_0,v_0,e_1,v_1,...,v_{r-1},e_r$ in $G$ is the sum $\sum_{k=1}^r(\Delta g(e_{k-1},v_{k-1},e_k)$.
\end{defn}

\begin{lem} \label{GradeRem}
1) If an ADG has a grading $g$, then the grading difference of a path $e_0,v_0,e_1,v_1,...,v_{r-1},e_r$ is equal to $g(e_r)-g(e_0)$.

2) An ADG is gradable iff it has no grade obstructing loop and, for every pair of edges $e$ and $f$, every path between $e$ and $f$ has the same grading difference.

3) One can check if an ADG $G$ is gradable, and therefore construct a grading, in linear $O(|E|)$ time where $E$ is the set of $G$'s edges.
\end{lem}

\begin{proof}
1) For a short part $e,v,f$ this follows directly from definitions \ref{GraphGrade} and \ref{RelGrade}(1). Induction implies the general case.

2) ($\Leftarrow$): The first part is remark \ref{GradeObsRem}, and the second follows from (1).

($\Rightarrow$): For every connected component $G'$ of $G$ (that is not a double circle), do the following: choose one edge $e$ in $G'$ and give it the grade $0$. Next, for every other edge $f$ in $G'$ choose a path $e=e_0,v_0,e_1,...,e_r=f$ and set the grade $g(f)$ of $f$ to be the relative grade of this path. By assumption this is independent of the path. If $f$ shares a vertex $v$ with another edge $h$, then $e=e_0,v_0,e_1,...,e_r=f,v,h$ is a path from $e$ to $h$, and so $g(h)=\sum_{k=1}^r(\Delta g(e_{k-1},v_{k-1},e_k)+\Delta g(f,v,h)=g(f)+\Delta g(f,v,h)$. This holds for every adjacent pair of edges. In particular, if $v$ is a vertex and $f$ is non-preferred at $v$, then, for the number $a(v)=g(f)$, every non-preferred edge $h$ at $v$ upholds $g(h)=g(f)+\Delta g(f,v,h)=a(v)$ and every preferred edge $h$ at $v$ upholds $g(h)=g(f)+\Delta g(f,v,h)=a(v)+1$, and so $g$ is a grading.

3) It takes $O(|E|)$ time to go over the edges of $G$ and check if any of them is a grade-obstructing loop. If no such loop exists, I will assign each edge $f$ of $G$ a number $g(f)$ which, if the graph is gradable, will be a grading. I say that the algorithm "reached" (resp. "exhausted") a vertex if it assigned a grading to at least one (resp. all) of the edges of this vertex. I begin by choosing one edge $e$ and grading it $g(e)=0$. For each vertex of $e$, I set $a(v)=g(e)-1=-1$ / $a(v)=g(e)=0$ if $e$ is respectively preferred / non-preferred at $v$.

Next, I choose a vertex $v$ that the algorithm has reached but has not exhausted (currently, this means that $v$ is one of the vertices of $e$) and go over the edges of $v$. If a preferred / non-preferred edge $f$ has yet to be graded, then grade it $g(f)=a(v)+1$ / $g(f)=a(v)$ respectively, then look at the other vertex $w$ of $f$. If this is the first time the algorithm reaches $w$, set $a(w)=g(f)-1$ / $a(w)=g(f)$ if $f$ is respectively preferred / non-preferred at $w$. If the algorithm reached $w$ before, then $a(w)$ has already been set previously. In order for $g$ to be a grading, $w$ must uphold $a(w)=g(f)-1$ / $a(w)=g(f)$, depending on if $f$ is preferred at $v$ or not. Check if this equality holds.

If the equality holds, move on to the other edges of $v$ and do the same. Since $v$ has no more than 6 edges, this takes $O(1)$ time. When you have exhausted $v$, move on to another vertex $G$ that the algorithm has reached but has yet to exhaust. Continue like this until either a) you grade an edge $f$ whose "other vertex" $w$ has already been reached and for which the appropriate equality $a(w)=g(f)-1$ / $a(w)=g(f)$ fails, or b) if you have not reached such an edge but there are no more vertices that the algorithm reached but has yet to exhaust.

If you stop because of (a) then $G$ is not gradable. In order to see this, notice that if you reached a vertex $v$ via an edge $e_v$, and then you grade another edge $f$ at $v$, then $g(f)=\Delta (g_v,v,f)+g(e_v)$. This can be proven on a case per case basis. For instance, if both $f$ and $e_v$ are preferred at $v$, then $\Delta (g_v,v,f)=0$ and according to the above $a(v)=g(e_v)-1$ and $g(f)=a(v)+1=g(e_v)=\Delta (g_v,v,f)+g(e_v)$ as required. Induction implies that every $f$ that the algorithm grades has a path $e=e_0,v_0,e_1,...,e_r=f$ such that $g(f)$ is equal to the grading difference of this path. Indeed, it holds for $e$ itself, and if you assume that it holds for every edge you graded before, and in particular for $e_v$, then $g(e_v)$ is equal to the grading difference of the path $e=e_0,v_0,e_1,...,e_r=g_v$ and $g(f)=g(e_v)+(g(f)-g(e_v))=\sum_{k=1}^r(\Delta g(e_{k-1},v_{k-1},e_k)+\Delta (g_v,v,f)$ - the grading difference of the path $e=e_0,v_0,e_1,...,e_r,v,f$.

Now, if you grade an edge $f$ whose other vertex $w$ has already been reached, and the appropriate equality $a(w)=g(f)-1$ / $a(w)=g(f)$ fails, then similar considerations imply that $g(f) \neq \Delta (g_w,w,f)+g(e_w)$. I have proven that there is one path from $e$ to $f$ whose grading difference is equal to $g(f)$, but there is a another such path $e=h_0,w_0,h_1,...,w_{r-1},h_r=e_w,w,f$, for which $g(e_w)=\sum_{k=1}^r(\Delta g(h_{k-1},w_{k-1},h_k)+\Delta (g_v,v,f)$ but $g(f)=g(e_w)+(g(f)-g(e_w)) \neq \sum_{k=1}^r(\Delta g(h_{k-1},w_{k-1},h_k)+\Delta (g_w,w,f)$. Since these two paths have different grading differences, (2) implies that $G$ is not gradable.

If the algorithm stopped because of (b) then it provided a grading $g(f)$ for every edge $f$ in the connected component of $G$ that contains $e$. Since the equality never failed, every vertex $v$ and every preferred / non-preferred edge $f$ at $v$ upholds $a(v)=g(f)-1$ / $a(v)=g(f)$. This means that $g$ is indeed a grading of this connected component. If there are any vertices left that the algorithm hasn't reached yet, then they belong to a different connected component. Choose a new ungraded edge $e$ and grade it $g(e)=0$, and then proceed to grade its connected component. Eventually, either you will reach stop condition (a), meaning that $G$ is not gradable, or you will exhaust all the vertices of $G$, in which case you finished grading all of $G$.

In total, the algorithm went over every edge $f$ of $G$, determined $g(f)$, and either determined $a(w)$ for one or both of its vertices, or checked if it upheld the equity $a(w)=g(f)-1$ / $a(w)=g(f)$. This takes $O(|E|)$ time.
\end{proof}

\begin{rem} \label{tree}
If the graph part of an ADG $G$ is a forest, then the algorithm will never reach the stop condition (a), and so $G$ is gradable.
\end{rem}

I can now formulate the main theorem:

\begin{thm} \label{ADGThm}
Let $M$ be an oriented 3-manifold for which $H(M;\Z)$ is periodic.

1) If $M$ has no boundary, then an ADG $G$ can be realized as the intersection graph of an oriented generic surface $i\co F \to M$ iff $G$ is gradable and has no DB values.

2) If $M$ has a boundary, then an ADG $G$ can be realized as the intersection graph of an oriented generic surface $i\co F \to M$ iff $G$ is gradable.
\end{thm}

\begin{res} \label{GradeRes}
In \cite{Li1}, Li showed that a DG with no DB values or branch values is realizable iff any arc in it is composed of an even number of edges. Theorem \ref{ADGThm} implies a generalization of this -  a general DG is realizable via an orientable generic surface iff every \underline{closed} arc is composed of an even number of edges.
\end{res}

\begin{proof} (of result \ref{GradeRes})
If a DG is realizable via an orientable generic surface, then any orientation of the surface gives the DG an ADG structure (arrows), and this ADG is realizable and therefore gradable. The grading of each subsequent edge on an arc will have a different parity than the grading of the previous edge and, in particular, closed arcs must have an even number of edges on them.

On the other hand, given a DG that upholds this condition (every closed arc must have an even number of edges), it is possible to give the DG a "short grading" - number the edges with only 0 and 1 in such a way that consecutive edges have different numbers. Clearly, the only obstruction to this is the existence of closed arcs with an odd number of edges. Now, one half-edge in every consecutive pair will belong to an edge whose grade is 1, and the other will belong to an edge whose grade is 0. You can give the DG an ADG structure that matches this grading by choosing the former half-edges to be preferred. This graded ADG is realizable, and in particular, the underlying DG is realizable via an orientable surface.
\end{proof}

In the remainder of this section I will prove the "only if" direction of the articles of theorem \ref{ADGThm}. The "if" direction will be proven in the next section. One part of the "only if direction" is trivial - a generic surface in a bounderyless 3-manifold cannot have DB values. In order to prove the other part, that the intersection graph of a generic surface is a gradable ADG, I use 3-dimensional winding numbers:

\begin{defn} \label{FaceBody}
Let $i\co F \to M$ be a proper generic surface in a 3-manifold $M$.

1) A face (resp. body) of $i$ is a connected component of $i(F) \setminus M(i)$ (resp. $M \setminus i(F)$).

2) Each face $V$ is an embedded surface in $M$, and there is a body on each side of it. I say these two bodies are adjacent (via $V$). A priori, it is possible that these two bodies are in fact two parts of the same body, and even that $V$ is a one-sided surface. In these cases, this body will be self adjacent, but this does not happen in any of the cases I am interested in.

3) If $i$ has a co-orientation, then each face $V$ is two sided, and the arrows on the face point towards one of its two sides. I say that the body on the side that the arrows point toward is "greater" (via $V$) than the body on the other side of $V$.

4) A choice of "winding numbers" for $i$ is a choice of an integer $w(U) \in \Z$, for every body $U$ of $i$, such that if $U_1$ and $U_2$ are adjacent, and $U_1$ the greater of the two, then $g(U_1)=g(U_2)+1$.
\end{defn}

\begin{lem} \label{Winding}
If $M$ is a connected and orientable 3-manifold, $H_1(M;\Z)$ is periodic, and $i\co F \to M$ is a co-oriented generic surface, then $i$ has a choice of winding numbers.
\end{lem}

\begin{proof}
Pick one body $U_0$ to be "the exterior" of the surface and set $w(U_0)=0$. Next, define the winding numbers for every other body $U$ like so:

Take a smooth path from $U_0$ to $U$ that is in general position to $i$ (it intersects $i(F)$ only at faces, and does so transversally), and set $w(U)$ to be the signed number of times it crosses $i(F)$, the number of times it intersects it in the direction of the co-orientation minus the times it crosses it against the co-orientation. This is well defined, since any two such paths $\alpha$ and $\beta$ must give the same number. Otherwise, the composition $\beta^{-1} \ast \alpha$ is a 1-cycle whose intersection number with the 2-cycle represented by $i$ is non-zero. This implies that this 1-cycle is of infinite order in $H_1(M;\Z)$ - contradicting the fact that this group is periodic.

It is also clear that if $U_1$ and $U_2$ are adjacent and $U_1$ is the greater of the pair, then $g(U_1)=g(U_2)+1$.
\end{proof}

\begin{rems} \label{WNRems}
1) It is clear that two different choices of "winding numbers" for $i$ will differ by a constant, and that the one I created is unique in satisfying $w(U_0)=0$.

2) I can do a similar process on a loop $\gamma$ in $\R^2$ instead of a surface in a 3-manifold. If I choose the component $U_0$ of $\R^2 \setminus Im(\gamma)$ to be the actual exterior, then this will produce the usual winding numbers - $w(U)$ will be the number of times $\gamma$ winds around a point in $U$. 
\end{rems}

I will use the winding numbers to induce a grading in the following manner: the neighborhood of a double value includes 4 bodies, with the possibility that some of them are, in fact, different parts of the same body. If the surface has a co-orientation and winding numbers, then there is a number $g$ such that two of these bodies have the WN $g$, one has the WN $g+1$ and one has the WN $g-1$. Figure \ref{fig:Figure 5}A depicts this:

\begin{figure} \label{fig:Figure 5}
\begin{center}
\includegraphics{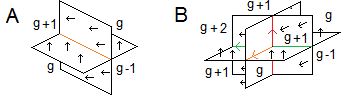}
\caption{The winding number of bodies around an edge of $M(i)$ and a triple value}
\label{fig:Figure 5}
\end{center}
\end{figure}

Due to continuity, this will be the same value $g$ for all the double values on the same edge (or double circle). I call this number the grading of the edge, and name the grading of an edge $e$ $g(e)$. This is indeed a grading in the sense of definition \ref{GraphGrade}. In order to prove this, I need to show that at every triple value of the surface all the preferred half-edges have the same grading, which is greater by 1 than the grading of all the non-preferred ones. This can be seen in figure \ref{fig:Figure 5}B, which depicts the winding numbers of the bodies around an arbitrary triple value. Indeed, you can see that the preferred half-edges - the ones going up, left and outwards (toward the reader) have the grading $g+1$, while the other edges have the grading $g$. This proves the "only if" direction of theorem \ref{ADGThm}.

\section{Realizing Graded Arrowed Daisy Graphs}

In order to prove the "if" direction of theorem \ref{ADGThm} I will first prove a partial result. I will limit the discussion to connected ADGs with no DB values.

\begin{lem} \label{NoDB}
Every \underline{connected}, \underline{gradable} ADG $G$ \underline{without DB values} has a closed generic surface $i\co F \to S^3$ such that the intersection graph of $i$ is equal, as an ADG, to $G$.
\end{lem}

\begin{rem} \label{Connected1}
It may be assumed that $i(F)$ is connected. Otherwise, one of its components will contain the connected intersection graph, and the rest will be embedded connected surfaces in $S^3$. They can removed them by deleting their preimages from $F$.
\end{rem}

I begin with the unique case where the ADG is a double circle. The generic surface from figure \ref{fig:Figure 6}A has a single double circle as its intersection graph. It is the surface of revolution of the curve from figure \ref{fig:Figure 6}B around the blue axis. Both figures have indication for the co-orientation. The intersection graph will be the revolution of the orange dot where the curve intersects itself, and will thus be a circle. The underlying surface is clearly a sphere.

\begin{figure}
\begin{center}
\includegraphics{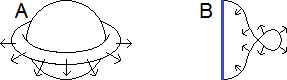}
\caption{A surface whose intersection graph is a double circle}
\label{fig:Figure 6}
\end{center}
\end{figure}

Any other connected ADG is a "graph ADG" - it will have no double circles. In this case, I begin by constructing a part of the matching generic surface - the regular neighborhood of the intersection graph. Li defined something similar in \cite{Li1} (p.3723, figure 2) which he called a "cross-surface", and I will use the same notation.

\begin{defn} \label{CrossSurf}
Given a ADG $G$ that has no DB values and no double circle, a "cross-surface" $X_G$ of $G$ is a shape in $S^3$ that is built via the following two steps:

1) For every triple value $v$ of $G$, embed a copy of figure \ref{fig:Figure 7}A in $S^3$. This shape is called the "vertex neighborhood" of $v$. Similarly, for every cross cap $v$ of $G$, embed a vertex neighborhood that looks like figure \ref{fig:Figure 7}B in $S^3$. Make sure that the different vertex neighborhoods will be pairwise disjoint. These will be the neighborhoods of the actual triple values and branch values of the surface we are constructing. They vertex neighborhoods have little arrows on them which indicate the co-orientation on this part of the surface. It is important to remember which vertex neighborhood corresponds to which vertex of $G$.

\begin{figure}
\begin{center}
\includegraphics{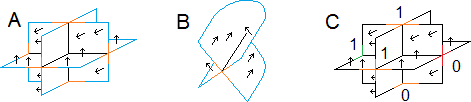}
\caption{The vertices neighborhoods and their gluing zones}
\label{fig:Figure 7}
\end{center}
\end{figure}

Each vertex $v$ is supposed to have either 1 or 6 half-edges enter it. One can see the ends of these half-edges in the illustrations. Note that each half-edge will cross the boundary of the vertex neighborhood at one point. Given such an intersection point, I refer to its regular neighborhood \underline{inside the boundary of the vertex neighborhood} as its "gluing zone". In figure \ref{fig:Figure 7}A and B, I colored the gluing zones in orange and the rest of the boundary the vertex neighborhoods in blue.

The above implies that each such "gluing zone" on the vertex neighborhood of $v$ should correspond to a unique half-edge of $G$ that ends in $v$. The reader has some freedom in choosing which gluing zone corresponds to which half-edge, but in accordance with definition \ref{Graph1}, the following must happen for every triple value, in order for this association to reflect the ADG structure of $G$:

(a) Two gluing zones on opposite sides of the value's neighborhood, like those marked red and green in figure \ref{fig:Figure 7}C, must correspond to a pair of consecutive half-edges.

(b) In compliance with the co-orientation (the little arrows) on the vertex neighborhoods, the zones that the arrows point toward - those marked with the number 1 in figure \ref{fig:Figure 7}C - must correspond to the preferred half-edges. The other zones, marked with 0, will correspond to the non-preferred half-edges. 

Again, it is important to remember which gluing zone corresponds to which half-edge.

\begin{figure}
\begin{center}
\includegraphics{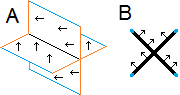}
\caption{The X-bundle of an edge and a cross-section of it}
\label{fig:Figure 8}
\end{center}
\end{figure}

2) Step 2 will realize the edges of $G$. The "ends" of each edge have already been realized inside the corresponding vertex neighborhoods. Each end is realized by the double line between the vertex and the corresponding gluing zone. I want to add the "length of the edge" to our construction. This should be a double line, the intersection of two surfaces, as in figure \ref{fig:Figure 8}A. For each edge $e$ of $G$, embed a matching copy of this shape in $S^3$. A closer look reveals that this shape is a bundle over a closed interval, whose fiber looks like the "X" in figure \ref{fig:Figure 8}B. I therefore call this shape "the X-bundle of $e$". The embedding of the X-bundles must follow the following rule:

\indent(a) The boundary of each X-bundle is composed of two parts - The fibers at the ends of the interval, colored orange, and the (union of the) ends of all the fibers, colored blue. It is natural to identify the two orange fibers with the two ends (half-edges) of the matching edge $e$. Until now, I identified each half-edge of the ADG with both i) a "gluing zone" on the boundary of the neighborhood of some vertex and ii) a fiber at the end of some X-bundle. Make sure to embed the X-bundles so that each end fiber coincides with the matching gluing zone. Additionally, make sure that the "length" of the X-bundle (the X-bundle sans the end fiber) is disjoint from the vertex neighborhoods, and that X-bundles of different edges do not touch one another.

The resulting shape is the cross-surface. It is similar to a generic surface but it has a boundary - the union of all the "blue parts" of the boundaries of the vertex neighborhoods and the X-bundles. The intersection graph of this "generic surface with a boundary" is clearly isomorphic (as a multigraph) to $G$ - I already identified each vertex / edge of it with a unique vertex / edge of $G$ and made sure that each edge ends in the vertices it should properly end in. Rule (a) from step (1) implies that this identification will preserve the consecutive pairs of half-edges. This means that the intersection graph is isomorphic to $G$ as a DG, not just as a multigraph. In order for this to be an ADG isomorphism as well, the embedding of the X-bundles must follow another rule:

(b) To have an ADG structure, the cross surface must have a co-orientation. Note that both the vertex neighborhoods and the X-bundles have arrows on them, which represent co-orientations. When you embed the X-bundles, these co-orientations on them must match, as in figure \ref{fig:Figure 9}A, and unlike figure \ref{fig:Figure 9}B. This way they will merge into a continuous co-orientation on the entire cross surface.

\begin{figure}
\begin{center}
\includegraphics{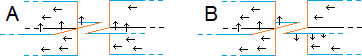}
\caption{The gluing must preserve the orientation}
\label{fig:Figure 9}
\end{center}
\end{figure}

Rule (b) from step (1) implies that the preferred half-edges of the intersection graph will correspond to the preferred half-edges of $G$. This means that the intersection graph will be isomorphic to $G$ as an ADG as well.
\end{defn}

The boundary of the cross-surface is the union of many embedded intervals in $S^3$ - the "blue parts" of the boundaries of the vertex neighborhoods and the X-bundles. Since each end of every interval coincides with an end of one other interval, and the intervals do not otherwise intersect, their union is an embedded compact 1-manifold in $S^3$. The cross-surface induces an orientation on this 1-manifold, the usual orientation that an oriented manifold induces on its boundary. It is depicted in the left part of figure \ref{fig:Figure 10}A. 

I will show that the boundary of the cross surface of a connected ADG $G$ is also the oriented boundary of an embedded surface which is disjointed from the cross surface. It follows that the union of the cross surface and the embedded surface, with the orientation on the embedded surface reversed, will be a closed and oriented generic surface whose intersection graph will be isomorphic to $G$. This will prove lemma \ref{NoDB}.

\begin{figure}
\begin{center}
\includegraphics{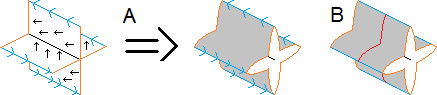}
\caption{Thickening the cross-surface into a handle body, and the handle body's meridians}
\label{fig:Figure 10}
\end{center}
\end{figure}

In order to prove that such an embedded surface exists, I begin by "thickening" the cross surface as in figure \ref{fig:Figure 10}A. Figure \ref{fig:Figure 10}A only shows how to do this to an X-bundle, but you can similarly do this for all the vertex neighborhoods. This results in a handle body $H$ in $S^3$ and our 1-cycle is on its boundary. It will suffice to prove that the 1-cycle is the boundary of some embedded surface in the complement of $H$. This happens iff is the cycle is a "boundary" in the homological sense - the is equal to $0$ in $H_1(\overline{S^3 \setminus H};\Z)$.

Given any loop $\gamma$ in the intersection graph, I define a functional $f_\gamma:H_1(\overline{S^3 \setminus H}$ $;\Z) \to \Z$ such that $f_\gamma(c)$ is the linking number of $\gamma$ and a representative of $c$. It is well-defined, since cycles in $\overline{S^3 \setminus H}$ are disjoint from $\gamma$, and since the linking number of $\gamma$ with any boundary in $H_1(\overline{S^3 \setminus H})$ is $0$, as the boundary bounds a surface in $\overline{S^3 \setminus H}$ which is disjoint from $\gamma$.

In case the genus of $G$, and therefore of the intersection graph and of $H$, is $n$, then the intersection graph has $n$ simple cycles $C_1,...,C_n$, such that each cycle $C_i$ contains an edge $e_i$ that is not contained in any of the other cycles. For every cycle $C_i$ I take a small meridian $m_i$ around the edge $C_i$ (as depicted in red in figure \ref{fig:Figure 10}B). It follows that $f_{c_i}([m_j])=\delta_{ij}$ where $\delta$ is the Kronecker delta function. Additionally, since $\overline{S^3 \setminus H}$ is the complement of an $n$-handle body, $H_1(\overline{S^3 \setminus H}) \equiv \Z^n$. I will prove that:

\begin{lem}
These meridians form a base of $H_1(\overline{S^3 \setminus H})$.
\end{lem}

\begin{proof}
First, I show that the meridians are independent. This is because a boundary in $\overline{S^3 \setminus H}$ would have $0$ as the linking number with every $c_i$, but the linking number of a non-trivial combination $x=\sum a_i[m_i]$ with any $c_j$ will be $a_j$, and for some $j$, $a_j \neq 0$. Second, notice that this implies that $N=Span_{\Z}\{[m_1],...,[m_n]\}$ is a maximal lattice in $H_1(\overline{S^3 \setminus H}) \equiv \Z^n$, and therefore has a finite index.

Third, had $N$ been a strict subgroup of $H_1(\overline{S^3 \setminus H};\Z)$, then there would be an element $y \in H_1(\overline{S^3 \setminus H};\Z) \setminus N$. Define $b_i=lk(y,c_i)$ and $y'=y-\sum_{i=1}^n b_i[m_i]$. $y'$ will have $0$ as the linking number with every $c_i$, but it will not belong to $N$. The finite index of $N$ implies that $ky' \in N$ for some $k$, but $lk(ky',c_i)=k \dot 0=0$ for all $i$, and thus $ky'=0$. This means that $y'$ is a non-zero element of finite order in $H_1(\overline{S^3 \setminus H};\Z) \equiv \Z^n$, but no such element exists.
\end{proof}

\begin{lem} \label{partial}
Let $G$ be a connected ADG that has no DB values, is not a double circle, and is gradable. Then the linking number of the boundary of its cross surface with any simple cycle in the intersection graph of this cross surface is $0$.
\end{lem}

\begin{proof}
%
Let $C$ be a simple cycle in the intersection graph. It is composed of distinct vertices and edges $e_0,v_1,e_1,v_2,...,v_n,e_n=e_0$. Each $v_i$ is a triple value, since it is not a degree-1 vertex. I will perturb $C$ until it's in general position to the cross surface and calculate the intersection number of the "moved $C$" with the cross-surface. This will be equal to the linking number of $C$ and the boundary of the cross-surface.

I begin by pushing each edge $e_i$ away from its matching X-bundle in a direction that agrees with the co-orientation on both of the surfaces that intersect in this X-bundle, as in figure \ref{fig:Figure 11}. 

\begin{figure}
\begin{center}
\includegraphics{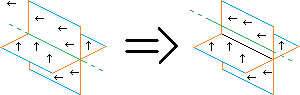}
\caption{Moving the intersection graph away from the cross surface}
\label{fig:Figure 11}
\end{center}
\end{figure}

I need to continue this "pushing" at the vertex neighborhood of each $v_i$. Figures \ref{fig:Figure 12}, \ref{fig:Figure 13} and \ref{fig:Figure 14} demonstrate how to push away the half-edges from their original position. The half-edges I push are colored green, and the arrows on them indicate the direction of the cycle - the half-edge whose arrow points toward (resp. away from) the triple value is a part of $e_{i-1}$ (resp. $e_i$). Continuity dictates that I must always push in the direction indicated by the orientations on the surface as we did in figure \ref{fig:Figure 11}, and figures \ref{fig:Figure 12}, \ref{fig:Figure 13} and \ref{fig:Figure 14} indeed comply with this.

Each of the three figures depict a different situation with regards to which of the two half-edges, if any, is preferred at $v_i$. Figure \ref{fig:Figure 12} depicts the case where both the half-edges are preferred. In this case, after being pushed away from the cross-surface, $C$ will not intersect the cross surface at the neighborhood of $v_i$.

\begin{figure}
\begin{center}
\includegraphics{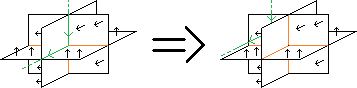}
\caption{Moving the intersection graph away from a triple value, when both sides are preferred}
\label{fig:Figure 12}
\end{center}
\end{figure}

Figure \ref{fig:Figure 13} depicts the case where the half-edge that is a part of $e_{i-1}$, the one entering the triple value, is not preferred, and the half-edge that is a part of $e_i$, the one exiting the triple value, is preferred. In this case, after being pushed away from the cross-surface, $C$ will intersect the cross-surface once, and it will do so agreeing with the direction of the co-orientation on the surface (that's why there is a little $+1$ next to the intersection.) 

\begin{figure}
\begin{center}
\includegraphics{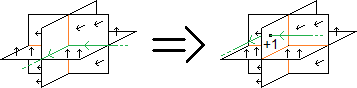}
\caption{Moving the intersection graph away from a triple value, when only one side is preferred}
\label{fig:Figure 13}
\end{center}
\end{figure}

Figure \ref{fig:Figure 13} depicts the case where the two half-edges are not consecutive, but even if they were, the same thing would happen - $C$ would intersect the cross-surface once, in agreement with the co-orientation. The only difference would be that the half-edge that was exiting $v_i$ would have continued leftwards instead of turning outwards towards the reader. Furthermore, had the half-edge coming from $e_{i-1}$ been preferred and the one coming from $e_i$ hadn't, then the pushing would still occur as in figure \ref{fig:Figure 13}, except that the arrows on the green line would point the other way. In this case $C$ would still intersect the cross surface once after the pushing, but it would be against the direction on the co-orientation.

Lastly, figure \ref{fig:Figure 14} depicts the case in which both half-edges are not preferred. In this case, after being pushed away from the cross-surface, $C$ will intersect the cross-surface twice in the neighborhood of $v_i$. One intersection, marked $+1$, is in the direction of the co-orientation, and the other intersection, marked $-1$, is against it.

\begin{figure}
\begin{center}
\includegraphics{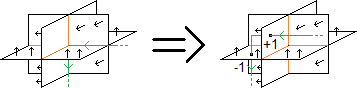}
\caption{Moving the intersection graph away from a triple value, when both sides are non-preferred}
\label{fig:Figure 14}
\end{center}
\end{figure}

Let $G$ be a grading of the intersection graph. Since $e_{i-1}$ and $e_i$ share a vertex, the difference between their grading is at most $1$. If $g(e_i)-g(e_{i-1})=1$ (resp. $-1$), then $e_i$ (resp $e_{i-1}$) is preferred and $e_{i-1}$ (resp. $e_i$) is not. I just showed that in this case the signed number of intersections between the "pushed away" $C$ and the cross-surface is $1$ (resp. $-1$). If $g(e_i)-g(e_{i-1})=0$ then either both $e_i$ and $e_{i-1}$ are preferred, in which cases $C$ does not intersect the cross-surface around $v_i$, or they are both non-preferred, in which case they intersect once with and once against the co-orientation.

In all cases the signed number of intersections between the pushed $C$ and the cross-surface around $v_i$ is equal to $g(e_i)-g(e_{i-1})$. The pushed $C$ does not intersect the cross-surface anywhere else, and so their intersection number is $\sum_{i=1}^n(g(e_i)-g(e_{i-1}))=g(e_n)-g(e_0)=0$. Since $C$ did not cross the the boundary of the cross surface during the pushing, this ($0$) is equal to the linking number of $C$ and the boundary.
\end{proof}

Having proven lemmas \ref{partial} and \ref{NoDB}, I can now prove the "if" direction of the articles of theorem \ref{ADGThm}:

\begin{proof}
1) Each connected component $G_k$ of $G$ is gradable and lacks DB values, and thus has a realizing surface $i_k:F_k \to S^3$. Simply remove a point from $S^3 \setminus i_k(F_k)$ to regard $i_k$ as a surface in $\R^3$, and embed these copies of $\R^3$ as disjoint balls in the interior of $M$.

2) If $G$ has no DB values the proof of (1) holds. Otherwise, define a new ADG $G'$ in which each DB value of $G$ is replaced with a branch value. Realize $G'$, via (1), with a closed generic surface $i\co F \to S^3$ for which $F$ is connected.

Take a small ball around each of the branch values that replaces a DB value of $G$, as in figure \ref{fig:Figure 15}A. Figure \ref{fig:Figure 15}B depicts the intersection of the surface with the boundary of the ball. It is an "8-figure" as in figure \ref{fig:Figure 15}C, and the orange dot (the intersection in the 8-figure) is the intersection of the boundary with the intersection graph. If you remove this ball from $S^3$, then instead of ending at the cross cap, the edge will end at the orange dot in the 8-figure, which will become a DB value. It follows that after removing all these balls, the intersection graph will be an ADG isomorphic to $G$.

\begin{figure}
\begin{center}
\includegraphics{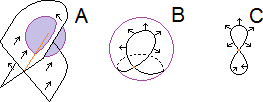}
\caption{Turning a branch value into a DB value}
\label{fig:Figure 15}
\end{center}
\end{figure}

The generic surface now lays in $S^3$ minus some number of balls. Choose one spherical boundary component and connect it via a path to each of the other ones. Make sure that the path is in general position to the generic surface - it may intersect it only at faces and will do so transversally. Thicken these paths into narrow 1-handles and remove them from the 3-manifold. This may remove some disc from the surface, but will not effect its intersection graph. You now have a generic surface that realizes $G$ in $D^3$. Remove a point from the boundary of $D^3$, making it diffeomorphic to the closed half space $\{(x,y,z) \in \R^3|z \geq 0\}$ which can be properly embedded in any 3-manifold with a boundary. This finishes the proof.
\end{proof}

\begin{rem} \label{Connected2}
If needed, you can make sure that the underlying surface $F$ is connected. This involves modifying the surface in two ways.

a) You can modify the proof of article (1) to produce a surface $i\co F \to M$ with a connect image.
Begin by assuming that the image of each $i_k$ is connected via remark \ref{Connected1}. Pick a face $v_k$ in each $i_k$. The co-orientation on $v_k$ points towards a body $U_k$. When you remove a point from $S^3$, make sure you remove it from $U_k$. This way, $U_k$ (minus a point) becomes the exterior body of $i_k:F_k \to \R^3$. When you embed the copies of $\R^3$ in $M$, the co-orientation on all $v_k$s will point towards the same connected component of $M \setminus \bigcup i_k(F_k)$. You may connect each $V_k$ to $V_{k+1}$ with a handle going through this component as in figure \ref{fig:Figure 18} (ignore the letters "A" and "B" in the drawing). This connects the images of the $i_k$s without sacrificing the orientation or changing the intersection graph.

In article (2) you take a surface from (1) and modify it. It is clear that none of these modifications can disconnect the image of the surface, and so (2) may also produce surfaces with a connected image.

b) In case $i(F)$ is connected but $F$ has more than one connected component then the images of some pair of connected components must intersect generically at a double line. This is depicted in the left part of figure \ref{fig:Figure 16}, where the vertical surface comes from one connected component of $F$ and the horizontal comes from another. Connect them via a handle in an orientation preserving way, as in the right part of figure \ref{fig:Figure 16}, thereby decreasing the number of connected components of $F$. Continue in this manner until $F$ is connected. 
\end{rem}

\begin{figure}
\begin{center}
\includegraphics{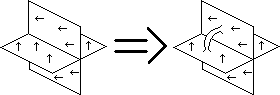}
\caption{Turning a disconnected surface into a connected one}
\label{fig:Figure 16}
\end{center}
\end{figure}

\section{Infinite homology}

In this section I deal with 3-manifold with whose first homology group contains an element of infinite order.

\begin{thm} \label{ADGThm2}
If $M$ is an oriented, compact and boundaryless 3-manifold with an infinite first homology group, then any ADG $G$ with no DB values can be realized as the intersection graph of an oriented generic surface in $M$. If $M$ has a boundary then any ADG $G$ can be realized in $M$.
\end{thm}

The proof relies on two lemmas:

\begin{lem} \label{Once}
$M$ has a connected, compact, oriented and properly embedded surface $S \subseteq M$ that is non-dividing ($M \setminus S$ is connected).
\end{lem}

\begin{proof}
$H_2(M;\Z)$ is generated by 2-cycles of the form $[S]$ where $S \subseteq M$ is a connected, compact, oriented and properly embedded surface. If the statement of the lemma is false, then each such surface divides $M$ into two connected components and will therefore be a boundary in $H_2(M;\Z)$. This implies that $H_2(M;\Z) \equiv \{0\}$. According to Poincar\'{e}'s duality, $\{0\} \equiv H_2(M;\Z)/Tor(H_2(M;\Z)) \equiv H_1(M;\Z)/Tor(H_1(M;\Z))$. This implies that every element of $H_1(M;\Z)$ is of finite order, contradicting the assumption.
\end{proof}

\begin{lem} \label{OneBody}
If $G$ is gradable, then there is a generic surface $i\co F \to M$, which realizes $G$, and for which $M \setminus i(F)$ is connected (equivalently, $i$ has only one body).
\end{lem}

\begin{proof}
Take the generic surface $S \subseteq M$ from lemma \ref{Once}, and a subset $M' \subseteq M$ that is disjoint from $S$ and is homomorphic to a half-space (if $M$ has a boundary) or to $\R^3$ (if it does not). According to theorem \ref{ADGThm}, there is a generic surface $i\co F \to M'$ which realizes $G$. Connect some face $V$ of the generic surface to $S$ with a handle, as in figure \ref{fig:Figure 17} (the handle does not intersect $i(F)$ or $S$). If needed, reverse the co-orientation of $S$ so that the resulting surface will be continuously co-oriented.

\begin{figure}
\begin{center}
\includegraphics{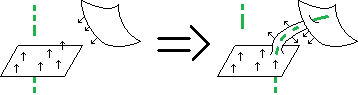}
\caption{Giving the surface a face that has the same body on both sides}
\label{fig:Figure 17}
\end{center}
\end{figure}

You now have a new generic surface $i':F \# S \to M$ whose intersection graph is still isomorphic to $G'$. Since $S$ was non-dividing, the connected sum of $V$ and $S$ is a face of this surface that has the same body $A$ on both sides (as indicated by the green path which does not intersect the surface in figure \ref{fig:Figure 17}). If this is the surface's only body then you are done. If not, you can decrease the number of bodies as follows:

Let $B$ be another body of the surface that is adjacent to $A$. Connect the face $W$ which separates $A$ and $B$ to the face $V \# S$ with a path that goes through $A$, and does not intersect our generic surface except at the ends of the path. Since $V \# S$ has $A$ on both sides, you can approach it from either side. If the arrows on $W$ points toward $A$ (resp. $B$), make sure the path enters $V \# S$ from the direction the arrows point towards (resp. point away from). Next, attach the faces $V$ and $W$ with a handle that runs along this path. Figure \ref{fig:Figure 18} depicts the case there the arrows on $W$ point towards $A$. Reverse the direction of all arrows to get the other case.

\begin{figure}
\begin{center}
\includegraphics{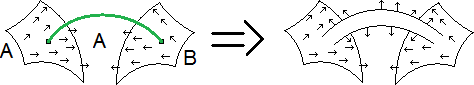}
\caption{Reducing the number of bodies}
\label{fig:Figure 18}
\end{center}
\end{figure}

The resulting generic surface has one body less since $A$ and $B$ have merged. It still realizes $G'$ and has a face with the same body on both sides. Repeat this process until you get a surface with only one body.
\end{proof}

I will now prove theorem \ref{ADGThm2}:

\begin{proof}
Let $H$ be the graph part of $G$ - $G$ without the double circles. I use induction on the genus of $H$. If the genus is $0$, then $G$ is the union of a forest with some double circles, and remark \ref{tree} implies that it is gradable and the theorem follows from lemma \ref{OneBody}. If the genus of $H$ is positive, pick an edge $e \in H$ such that $H \setminus \{e\}$ has a smaller genus. This means that removing $e$ does not divide the connected component of $H$ that contains $e$. Note that both ends of $e$ are on triple values, since branch values and DB values are of degree 1 and removing their single edge divides the graph.

Define a new ADG $G'$ in the following manner: start with a copy of $G$ and cut the edge $e$ in the middle. Instead of $e$ you will get two "new edges" $e_1$ and $e_2$. Each $e_i$ has one end on a new branch value while the other end "replaces" one of the ends of $e$ - it enters the triple value that the said end of $e$ was on, and it retains the ADG data - it is preferred iff the half-edge of $e$ was preferred, and it has the same consecutive half-edge. Figure \ref{fig:Figure 19} depicts the two possible ways to construct $G'$ from $G$. 

\begin{figure}
\begin{center}
\includegraphics{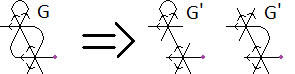}
\caption{Cutting an edge and adding two branch values to an arrowed daisy graph}
\label{fig:Figure 19}
\end{center}
\end{figure}

$H'$, the graph structure of $G'$, has a lower genus then $H$. I assume, by induction, that there is a generic surface in $M$ that realizes $G'$ and \underline{has only one body}. I will modify this surface so that it realizes $G$. Observe the new branch values at the ends of $e_1$ and $e_2$. Change the surface in a small neighborhood of each branch value as per figure \ref{fig:Figure 20}A, deleting the branch value and leaving instead a "figure 8 boundary" of the surface.

\begin{figure}
\begin{center}
\includegraphics{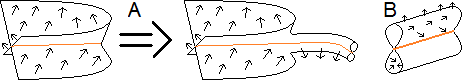}
\caption{Removing two branch values from a generic surface and restoring the previously cut edge}
\label{fig:Figure 20}
\end{center}
\end{figure}

This figure 8 boundary is depicted in figure \ref{fig:Figure 15}C. Take a bundle over an interval whose fibers are "8-figures", as in figure \ref{fig:Figure 20}B, and embed it in $M$ in such a way that its end-fibers coincide with the said "figure 8 boundaries" (in a way that preserves the arrows of the co-orientation). Since the complement of the original surface was connected, you can make sure that the bundle does not intersect the surface anywhere except its ends. This closes $e_1$ and $e_2$ into one edge, reversing the procedure that created $G'$ from $G$, and so this new surface realizes $G$ while still having only one body. The proof follows by induction.
\end{proof}

\begin{rem} \label{Connected3}
It is possible once more to make sure that the underlying surface $F$ is connected. Firstly, you may connect the different connected components of $i(F)$ via handles, similarly to the way you connected faces in the proof of lemma \ref{OneBody}. You may then proceed as in lemma \ref{Connected2}(b).
\end{rem}

\section{Ordered daisy graphs}

In the last section, I will refine the enhanced graph structure of the intersection graph of a generic surface from an ADG to a structure I call an ordered daisy graph, or ODG, which encodes more information regarding the topology of the surface.

For motivation, attempt to construct a cross-surface for an ADG $G$ as per definition \ref{CrossSurf}. Figure \ref{fig:Figure 21}A depicts a vertex neighborhood of a triple value $v$. Its preferred half-edges are indexed as $+1,+2,+3$ (and the corresponding non-preferred ones as $-1,-2,-3$). While constructing the cross-surface, you glue the end of some X-bundle to each of these half-edges. Let $\sigma \in S_3$ be the even permutation $(1,2,3)$. If, for each $k=1,2,3$, you take the "end of an X-bundle" that is supposed to be glued to the half-edge $\pm k$, and instead glue it to the half-edge $\pm \sigma(k)$, then you would end up with essentially the same cross-surface.

\begin{figure}
\begin{center}
\includegraphics{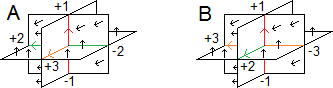}
\caption{Even and odd permutations of a triple value neighborhood}
\label{fig:Figure 21}
\end{center}
\end{figure}

By "essentially the same cross-surface", I mean that you could pick neighborhoods $H_1 \subseteq S^3$ of one cross-surface and $H_2 \subseteq S^3$ of the other so that there is an orientation preserving homeomorphism $f:H_1 \to H_2$ that sends the first cross-surface to the second one in a manner preserving the co-orientation on them. In general, It may be impossible to extend $f$ to all of $S^3$, since the way to embed the X-bundles in $S^3$ in definition \ref{CrossSurf} allows the reader to "knot" the cross-surface as they see fit. In this case, $f$ acts as a rotation in the neighborhood of $v$ (the rotation that sends figure \ref{fig:Figure 21}A to figure \ref{fig:Figure 21}B). After the rotation, it may need to move the X-bundles of the cross-surface in order to knot them in a different way.

Using an odd permutation instead of $\sigma$ might produce a fundamentally different cross-surface. While I can define a similar $f$ via a reflection on the neighborhood of $v$ (such as the reflection that sends figure \ref{fig:Figure 21}A to figure \ref{fig:Figure 21}C), this $f$ does not preserve the orientation on $H_1$. Furthermore, the two cross-surfaces may not even be homeomorphic subsets of $S^3$. It's possible to complete each cross-surface into a generic surface. This will produce two generic surfaces $i_k:F_k \to S^3$ ($k=1,2$) such that, on the one hand $M(i_1)$ and $M(i_2)$ are isomorphic as ADGs but on the other hand the neighborhoods of $M(i_1)$ and $M(i_2)$ are topologically distinct. I refine the enhanced graph structure of the intersection graph so that it reflects the difference between them.

\begin{defn}
1) I add, as a 7th entry to the 6-tuple of the ADG, a function that assigns each triple value with an indexing of its preferred edges, which is unique up to an even permutation, with which the triple value looks like figure \ref{fig:Figure 21}A (and unlike figure \ref{fig:Figure 21}C). I refer to this as the Ordered Daisy Graph (ODG) structure of the intersection graph.

2) A cross-surface of an ODG $G$ is defined similarly to a cross-surface of an ADG (definition \ref{CrossSurf}) with the additional requirements that the vertex neighborhood of each triple value comes with indexed half-edges as per figure \ref{fig:Figure 21}A and that, when you attach the X-bundles to such a vertex neighborhood you comply with the indexing of the ODG structure.
\end{defn}

The benefit of using ODGs is that all the cross-surfaces one might produce for the same ODG are essentially the same. All of the results in the previous sections, and in particular theorems \ref{ADGThm} and \ref{ADGThm2}, can be rephrased to use ODGs instead of ADGs and the same proofs will still apply.
 
\section*{Acknowledgments}

I would like to thank Dr. Tahl Nowik, my Ph.D adviser, for introducing me to this subject and to the wonderful and intriguing field of geometric topology. And my wife, Hila, for editing my work.

\bibliographystyle{plain}
\bibliography{MyBib}


\end{document}